\documentclass[12pt]{article}
\usepackage[T1]{fontenc}
\usepackage{amsmath,amssymb,amsthm,amscd,dsfont}
\usepackage{amsfonts,latexsym,rawfonts,amsmath,amssymb,amsthm}
\usepackage{lscape}
\usepackage{amscd, float,times,rotating}
\usepackage{pb-diagram}
\usepackage[hyperfootnotes=false]{hyperref}
\hypersetup{hidelinks,pdftitle={Conformal Ricci Curvature and Spectral Estimates},pdfauthor={Xiaoshang Jin}}
\numberwithin{equation}{section}

\newtheoremstyle{refplain}
  {4pt}{4pt}
  {\itshape}
  {}
  {\bfseries}
  {.}
  {.5em}
  {}
\theoremstyle{refplain}
\newtheorem{theorem}{Theorem}[section]
\newtheorem{proposition}[theorem]{Proposition}
\newtheorem{lemma}[theorem]{Lemma}

\theoremstyle{remark}
\newtheorem{remark}[theorem]{Remark}

\newcommand{\Ric}{\operatorname{Ric}}
\newcommand{\dd}{\,\mathrm{d}}
\newcommand{\inrad}{\operatorname{inrad}}

\begin{document}

\title{Conformal Ricci Curvature and Spectral Estimates}
\author{Xiaoshang Jin\thanks{The author was supported by the Fundamental Research Funds for the Central Universities, HUST (No.~2025BRSXB002), and the National Natural Science Foundation of China (Grant No.~12471054).}}
\date{}
\maketitle

\begin{abstract}
We derive optimized upper bounds for the first Dirichlet eigenvalue of a bounded domain under lower bounds for the conformal Ricci tensor of Shen and Ye. The method also yields sharp estimates for the bottom spectrum on complete noncompact manifolds, a four-dimensional application involving $Q$-curvature, and local and global spectral-Ricci extensions of Cheng's estimate. As a consequence, we obtain a sharp comparison between the spectra of the Laplacian and the conformal Laplacian in terms of the smallest Schouten eigenvalue.
\end{abstract}

\noindent\textit{2020 Mathematics Subject Classification.} Primary: 58J50; Secondary: 53C21, 35P15.\\
\textit{Key words and phrases.} conformal Ricci curvature, spectral Ricci curvature, Dirichlet eigenvalue, conformal Laplacian, $Q$-curvature.

\section{Introduction}
Let $(M^n,g)$ be a Riemannian manifold, and let $\Omega\Subset M$ be a smooth, bounded, connected domain. A number $\lambda\in\mathbb R$ is a Dirichlet eigenvalue of $-\Delta_g$ on $\Omega$ if there exists a nonzero function $u\in C^\infty(\Omega)\cap C^0(\overline\Omega)$ satisfying
\[
\begin{cases}
-\Delta_g u=\lambda u & \text{in }\Omega,\\
u=0 & \text{on }\partial\Omega.
\end{cases}
\]
Here and below $\Delta_g=\operatorname{div}_g\nabla$. The first Dirichlet eigenvalue is positive and simple, its first eigenfunction can be chosen positive in $\Omega$, and
\[
\lambda_1(\Omega)
=
\inf_{0\neq\varphi\in C_c^\infty(\Omega)}
\frac{\displaystyle\int_\Omega |\nabla\varphi|^2\dd V_g}
{\displaystyle\int_\Omega \varphi^2\dd V_g}
=
\inf_{0\neq\varphi\in H_0^1(\Omega)}
\frac{\displaystyle\int_\Omega |\nabla\varphi|^2\dd V_g}
{\displaystyle\int_\Omega \varphi^2\dd V_g}.
\]

If $(M^n,g)$ is complete and noncompact and
\[
\Omega_1\Subset\Omega_2\Subset\cdots\Subset M,
\qquad
\bigcup_{j=1}^{\infty}\Omega_j=M,
\]
is a smooth exhaustion by connected domains, then domain monotonicity gives
\begin{equation}\label{eq:bottom-def}
\lambda_1(-\Delta_g)
:=\lim_{j\to\infty}\lambda_1(\Omega_j)
=
\inf_{0\neq\varphi\in C_c^\infty(M)}
\frac{\displaystyle\int_M|\nabla\varphi|^2\dd V_g}
{\displaystyle\int_M\varphi^2\dd V_g}
=\inf\sigma(-\Delta_g).
\end{equation}
Thus $\lambda_1(-\Delta_g)$ is the bottom of the spectrum of the nonnegative self-adjoint Laplace--Beltrami operator; see \cite{Chavel,LiBook}.

A fundamental theorem of Cheng \cite{Cheng} states that, on a complete noncompact manifold, if $\kappa\geq0$, then
\[
\Ric_g\geq-\kappa g
\quad\Longrightarrow\quad
\lambda_1(-\Delta_g)\leq\frac{n-1}{4}\,\kappa,
\]
and the constant is sharp on hyperbolic space. Replacing the Ricci lower bound by scalar-curvature information is substantially more delicate. Sharp estimates under additional topological or spin assumptions were obtained by Davaux \cite{Davaux} and, in dimension three, by Munteanu and Wang \cite{MW2024,MW2026}. Spectral versions of Ricci lower bounds, formulated through Schr\"odinger operators involving the smallest Ricci eigenvalue, have also been studied in \cite{CarronRose,AXBG,APX,CMMR,HongWang}. In particular, operators of the form $-\alpha\Delta_g+\Ric(x)$ appear explicitly in recent sharp splitting and boundary-splitting results \cite{APX,HongWang}.

For a positive smooth function $f$ and $\sigma>0$, Shen and Ye \cite{ShenYe} introduced
\begin{equation}\label{eq:conf-ric}
\Ric_g^{f,\sigma}:=\Ric_g-\sigma f^{-1}(\Delta_g f)g.
\end{equation}
We also allow the degenerate case $\sigma=0$, for which $\Ric_g^{f,0}=\Ric_g$. It is naturally associated with the conformal metric $f^{2\sigma}g$. For a smooth bounded connected domain $\Omega\Subset M$, we write
\begin{equation}\label{eq:inradius-def}
r_\Omega:=\inrad(\Omega):=\sup_{x\in\Omega}d_g(x,\partial\Omega).
\end{equation}
Our main result is a Dirichlet eigenvalue estimate under a lower bound for this tensor.

\begin{theorem}\label{thm:domain}
Let $\Omega\Subset M$ be a smooth bounded connected domain in an $n$-dimensional Riemannian manifold, $n\geq2$. Suppose that there exist
$\sigma\in[0,\frac{4}{n-1}),$ $f\in C^\infty(\Omega)$ with $f>0$ and $\kappa\in\mathbb R$ such that
\[
\Ric_g^{f,\sigma}\geq-\kappa g
\quad\text{on }\Omega.
\]
Then
\begin{equation}\label{eq:domain-main}
\lambda_1(\Omega)
\leq
\frac{n-1}{D_\sigma}
\left[
\sqrt{
\kappa+
\frac{4\pi^2\bigl((n-1)-(n-2)\sigma\bigr)}
{D_\sigma r_\Omega^2}
}
+
\frac{2\pi|n-3|}
{r_\Omega\sqrt{(n-1)D_\sigma}}
\right]^2.
\end{equation}
where $D_\sigma:=4-(n-1)\sigma.$
The quantity under the square root is nonnegative under the above assumptions.
\end{theorem}

\begin{remark}\label{rem:consequences}
Theorem \ref{thm:domain} has the following immediate consequences and limiting cases.
\begin{enumerate}\renewcommand{\labelenumi}{(\roman{enumi})}
\item If $\kappa=0$, namely $\Ric_g^{f,\sigma}\geq0$ on $\Omega$, then
\[
\lambda_1(\Omega)
\leq
\frac{4\pi^2}{D_\sigma^2r_\Omega^2}
\left[
\sqrt{(n-1)\bigl((n-1)-(n-2)\sigma\bigr)}
+|n-3|
\right]^2.
\]

\item If $f$ is constant, then $\Ric_g^{f,\sigma}=\Ric_g$ for every $\sigma$. Taking $\sigma=0$, if $\Ric_g\geq-\kappa g$ on $\Omega$, then
\begin{equation}\label{eq:ricci-domain}
\lambda_1(\Omega)
\leq
\frac14
\left[
\sqrt{(n-1)\kappa+\frac{(n-1)^2\pi^2}{r_\Omega^2}}
+\frac{\pi|n-3|}{r_\Omega}
\right]^2.
\end{equation}

\item If $(M^n,g)$ is complete and noncompact and the assumptions of Theorem \ref{thm:domain} hold globally with $\kappa\geq0$, then a smooth exhaustion with inradii tending to infinity gives
\begin{equation}\label{eq:main-bottom}
\lambda_1(-\Delta_g)\leq\frac{\kappa}{\frac{4}{n-1}-\sigma}.
\end{equation}

\item The coefficient in \eqref{eq:main-bottom} is sharp for every fixed $0\leq\sigma<\frac{4}{n-1}$. On $\mathbb H^n(-a^2)$, let $b$ be a Busemann function satisfying $|\nabla b|=1$ and $\Delta b=(n-1)a$, and set
\[
f=\exp\left(-\frac{(n-1)a}{2}b\right).
\]
Then
\[
\frac{\Delta f}{f}=-\frac{(n-1)^2a^2}{4},
\qquad
\lambda_1\bigl(-\Delta_{\mathbb H^n(-a^2)}\bigr)=\frac{(n-1)^2a^2}{4},
\]
and
\[
\Ric^{f,\sigma}
=-\left(\frac{4}{n-1}-\sigma\right)
\frac{(n-1)^2a^2}{4}\,g.
\]
Thus equality holds in \eqref{eq:main-bottom}. At $\sigma=\frac{4}{n-1}$ the same example has $\Ric^{f,\sigma}=0$ and $\lambda_1(-\Delta_g)>0$, so the range of $\sigma$ is optimal. When $f$ is constant and $\sigma=0$, \eqref{eq:main-bottom} reduces exactly to Cheng's estimate.
\end{enumerate}
\end{remark}

In dimension four, let
\[
A_g=\frac{1}{2}\left(\Ric_g-\frac{R_g}{6}g\right)
\]
be the Schouten tensor. Chang--Gursky--Yang proved \cite[Lemma 1.2]{CGY} that
\[
\Ric_g\geq12\frac{\sigma_2(A_g)}{R_g}g
\qquad\text{when }R_g>0.
\]
Since the four-dimensional $Q$-curvature satisfies
\[
Q_g=-\frac{1}{6}\Delta_gR_g+4\sigma_2(A_g),
\]
see also \cite{Branson1985}, it follows that
\begin{equation}\label{eq:Q-conformal-Ricci}
\Ric_g-\frac{\Delta_gR_g}{2R_g}g
\geq3\frac{Q_g}{R_g}g.
\end{equation}
This observation leads to the following application of Theorem \ref{thm:domain} and the global estimate \eqref{eq:main-bottom}.

\begin{theorem}\label{thm:Q}
Let $\Omega\Subset M^4$ be smooth, bounded, and connected. Assume that $R_g>0$ and $\frac{Q_g}{R_g}\geq-\kappa$ on $\Omega$ for some $\kappa\in\mathbb R$. Then
\begin{equation}\label{eq:Q-domain}
\lambda_1(\Omega)
\leq
\frac{1}{25}
\left[
\sqrt{90\kappa+\frac{96\pi^2}{r_\Omega^2}}
+\frac{4\pi}{r_\Omega}
\right]^2.
\end{equation}
If $(M^4,g)$ is complete and noncompact, $R_g>0$ on $M$, $\kappa\geq0$, and $\frac{Q_g}{R_g}\geq-\kappa$ globally, then
\begin{equation}\label{eq:Q-bottom}
\lambda_1(-\Delta_g)\leq\frac{18}{5}\kappa.
\end{equation}
\end{theorem}

In particular, if $R_g>0$ and $\lambda_1(-\Delta_g)>0$, \eqref{eq:Q-bottom} also gives the obstruction
\[
\inf_M\frac{Q_g}{R_g}
\leq-\frac{5}{18}\lambda_1(-\Delta_g).
\]

We next formulate an intrinsic spectral consequence of the conformal-Ricci estimate.  For each $x\in M$, define
\[
\Ric(x):=\min\bigl\{\Ric_g(X,X):X\in T_xM,\ |X|_g=1\bigr\}.
\]
Thus $\Ric(x)$ is the smallest Ricci eigenvalue at the point $x$.  If $\Omega\Subset M$ is a smooth bounded connected domain and $\alpha\geq0$, define
\begin{equation}\label{eq:spectral-ricci-domain-def}
\lambda_1\bigl(-\alpha\Delta_g+\Ric(x),\Omega\bigr)
:=
\inf_{0\neq\varphi\in C_c^\infty(\Omega)}
\frac{\displaystyle
\alpha\int_\Omega|\nabla\varphi|^2\dd V_g
+\int_\Omega\Ric(x)\varphi^2\dd V_g}
{\displaystyle\int_\Omega\varphi^2\dd V_g}.
\end{equation}
For a complete noncompact manifold, we also write
\begin{equation}\label{eq:spectral-ricci-bottom}
\lambda_1\bigl(-\alpha\Delta_g+\Ric(x)\bigr)
:=
\inf_{0\neq\varphi\in C_c^\infty(M)}
\frac{\displaystyle
\alpha\int_M|\nabla\varphi|^2\dd V_g
+\int_M\Ric(x)\varphi^2\dd V_g}
{\displaystyle\int_M\varphi^2\dd V_g}
\in\mathbb R\cup\{-\infty\}.
\end{equation}

\begin{theorem}\label{thm:spectral-cheng}
Let $\Omega\Subset M^n$ be a smooth bounded connected domain, $n\geq2$. For every
\[
0\leq\alpha<\frac{4}{n-1},
\qquad
0<\beta<\frac{4}{n-1}-\alpha,
\]
one has
\begin{equation}\label{eq:spectral-domain}
\lambda_1\bigl(-\alpha\Delta_g+\Ric(x),\Omega\bigr)
+\beta\lambda_1(\Omega)\leq
\frac{\pi^2}{r_\Omega^2}
\left(
 n-1+
 \frac{(n-3)^2}{\frac{4}{\alpha+\beta}-n+1}
\right).
\end{equation}
Consequently, if $(M^n,g)$ is complete and noncompact, then for every
$0\leq\alpha<\frac{4}{n-1}$,
\begin{equation}\label{eq:spectral-cheng}
\lambda_1\bigl(-\alpha\Delta_g+\Ric(x)\bigr)
\leq
-\left(\frac{4}{n-1}-\alpha\right)\lambda_1(-\Delta_g)
\leq0.
\end{equation}
The inequalities in \eqref{eq:spectral-cheng} are understood in the extended real sense.
\end{theorem}

\begin{remark}\label{rem:spectral-cheng}
The global estimate \eqref{eq:spectral-cheng} has the following three useful interpretations.  It is sharp for every admissible $\alpha$; equality holds on both Euclidean space and hyperbolic space.
\begin{enumerate}\renewcommand{\labelenumi}{(\roman{enumi})}
\item For $\alpha=0$, the first spectral term in \eqref{eq:spectral-cheng} is $\inf_M\Ric(x)$, and \eqref{eq:spectral-cheng} becomes
\[
\lambda_1(-\Delta_g)
\leq
-\frac{n-1}{4}\inf_M\Ric(x).
\]
In particular, if $\Ric_g\geq-\kappa g$, this is exactly Cheng's estimate
$\lambda_1(-\Delta_g)\leq\frac{n-1}{4}\kappa$.

\item For $0<\alpha<\frac{4}{n-1}$, \eqref{eq:spectral-cheng} is a spectral-Ricci extension of Cheng's estimate.  The condition
\[
\lambda_1\bigl(-\alpha\Delta_g+\Ric(x)\bigr)\geq0
\]
is commonly referred to as nonnegative Ricci curvature in the spectral sense; see, for example, \cite{APX,CMMR,HongWang}.  In the complete noncompact setting of Theorem \ref{thm:spectral-cheng}, \eqref{eq:spectral-cheng} shows that this nonnegativity is necessarily an equality case:
\[
\lambda_1\bigl(-\alpha\Delta_g+\Ric(x)\bigr)=0,
\qquad
\lambda_1(-\Delta_g)=0.
\]
The cited works use spectral nonnegativity as a curvature hypothesis in splitting, criticality, and boundary-splitting problems; the conclusion above records its spectral consequence in the subcritical range $0<\alpha<\frac{4}{n-1}$ for complete noncompact manifolds.

\item Suppose $n\geq4$ and take $\alpha=\frac{2}{n-2}$.  Let
\[
L_g:=-\frac{4(n-1)}{n-2}\Delta_g+R_g
\]
be the conformal Laplacian; $\lambda_1(L_g)$ denotes its quadratic-form spectral bottom. Let
\[
A_g:=\frac1{n-2}\left(\Ric_g-\frac{R_g}{2(n-1)}g\right)
\]
be the Schouten tensor.  For each $x\in M$, define
\[
\lambda_{\min}(A_g)(x)
:=\min\bigl\{A_g(X,X):X\in T_xM,\ |X|_g=1\bigr\}.
\]
Since
\[
\Ric(x)=(n-2)\lambda_{\min}(A_g)(x)+\frac{R_g(x)}{2(n-1)},
\]
one has the operator identity
\[
-\frac{2}{n-2}\Delta_g+\Ric(x)
=
(n-2)\lambda_{\min}(A_g)(x)+\frac{1}{2(n-1)}L_g.
\]
Consequently, by the variational characterization of the spectral bottom,
\[
\lambda_1\!\left(-\frac{2}{n-2}\Delta_g+\Ric(x)\right)
\geq
(n-2)\inf_{x\in M}\lambda_{\min}(A_g)(x)
+\frac{1}{2(n-1)}\lambda_1(L_g).
\]
Combining this inequality with Theorem \ref{thm:spectral-cheng} gives
\begin{equation}\label{eq:schouten-spectrum}
\lambda_1(-\Delta_g)
+\frac{n-2}{4(n-3)}\lambda_1(L_g)
\leq
-\frac{(n-1)(n-2)^2}{2(n-3)}
\inf_{x\in M}\lambda_{\min}(A_g)(x).
\end{equation}
The inequality is understood in the extended real sense and is also sharp: equality holds on $\mathbb R^n$ and on $\mathbb H^n(-a^2)$.
\end{enumerate}
\end{remark}

The paper is organized as follows. Section 2 recalls the Shen--Ye geodesic length estimate and derives the boundary-distance consequence needed later. Section 3 proves the Dirichlet estimate and obtains the bottom-spectrum estimate by exhaustion. Section 4 treats the four-dimensional $Q$-curvature application. Section 5 proves the bounded-domain estimate \eqref{eq:spectral-domain} and derives the global spectral-Ricci extension \eqref{eq:spectral-cheng} by exhaustion.

\section{Preliminaries}

\subsection{Conformal Ricci curvature and the Shen--Ye geodesic estimate}

Let $u>0$ and $0<\tau<\frac{4}{n-1}$, and set $\widetilde g=u^{2\tau}g$. The geodesic estimate needed below is Lemma 2 of Shen--Ye \cite{ShenYe}. In the notation used here, if a $\widetilde g$-geodesic $\alpha:[0,l]\to M$, reparametrized by $g$-arc length, minimizes $\widetilde g$-length to second order under fixed-endpoint variations and
\[
\Ric_g^{u,\tau}\geq\eta g
\]
along the curve for some $\eta>0$, then
\begin{equation}\label{eq:SY-length}
l
\leq
\frac{\pi}{\sqrt\eta}
\left(
 n-1+\frac{(n-3)^2}{\frac{4}{\tau}-n+1}
\right)^{\frac{1}{2}}.
\end{equation}

\subsection{A boundary-distance consequence}

\begin{proposition}\label{prop:radius}
Let $\Omega\Subset M$ be a connected smooth domain in an $n$-dimensional Riemannian manifold, $n\geq2$. Suppose that $u\in C^\infty(\Omega)$ is positive, that $0<\tau<\frac{4}{n-1}$, and that
\[
\Ric_g^{u,\tau}\geq\eta g
\]
for some $\eta>0$. Then
\begin{equation}\label{eq:radius}
\inrad(\Omega)
\leq
\frac{\pi}{\sqrt\eta}
\left(
 n-1+\frac{(n-3)^2}{\frac{4}{\tau}-n+1}
\right)^{\frac{1}{2}}.
\end{equation}
\end{proposition}

\begin{proof}
Fix $x\in\Omega$ and choose a nested smooth connected exhaustion
\[
\Omega_1\Subset\Omega_2\Subset\cdots\Subset\Omega,
\qquad x\in\Omega_1,
\qquad \bigcup_{j=1}^\infty\Omega_j=\Omega.
\]
Since $u$ is smooth and positive on $\Omega$, the conformal metric
\[
\widetilde g=u^{2\tau}g
\]
is smooth and nondegenerate on a neighborhood of $\overline{\Omega_j}$ for every $j$. Choose a complete metric on $M$ that agrees with $\widetilde g$ on such a neighborhood. Let $y_j\in\partial\Omega_j$ minimize the resulting distance from $x$ to $\partial\Omega_j$, and let $\alpha_j$ be a minimizing geodesic from $x$ to $y_j$. The geodesic remains in $\Omega_j$ except at its endpoint: otherwise its first intersection with $\partial\Omega_j$ would be closer to $x$ than $y_j$. Reparametrize $\alpha_j$ by $g$-arc length on $[0,l_j]$. It is minimizing to second order under fixed-endpoint variations, and hence the lower bound $\Ric_g^{u,\tau}\geq\eta g$ and Lemma 2 of \cite{ShenYe} imply
\begin{equation}\label{eq:curve-bound}
l_j
\leq
\frac{\pi}{\sqrt\eta}
\left(
 n-1+\frac{(n-3)^2}{\frac{4}{\tau}-n+1}
\right)^{\frac{1}{2}}.
\end{equation}
Because $\alpha_j$ joins $x$ to $\partial\Omega_j$,
\[
d_g(x,\partial\Omega_j)\leq l_j.
\]

We claim that
\[
d_g(x,\partial\Omega_j)\longrightarrow d_g(x,\partial\Omega).
\]
Indeed, the left-hand side is nondecreasing and bounded above by $d_g(x,\partial\Omega)$.  If its limit were strictly smaller, there would be $\delta>0$ and points $z_j\in\partial\Omega_j$ with
\[
d_g(x,z_j)\leq d_g(x,\partial\Omega)-\delta.
\]
After passing to a subsequence, compactness of $\overline\Omega$ gives $z_j\to z\in\overline\Omega$.  The point $z$ cannot lie in $\Omega$: otherwise $z$ would belong to the interior of some $\Omega_N$, and then $z_j\in\partial\Omega_j$ could not converge to $z$ for $j\geq N$.  Hence $z\in\partial\Omega$, contradicting
\[
d_g(x,z)\leq d_g(x,\partial\Omega)-\delta.
\]
Thus the claimed convergence holds.  Letting $j\to\infty$ in the preceding inequalities gives
\[
d_g(x,\partial\Omega)
\leq
\frac{\pi}{\sqrt\eta}
\left(
 n-1+\frac{(n-3)^2}{\frac{4}{\tau}-n+1}
\right)^{\frac{1}{2}}.
\]
Taking the supremum over $x\in\Omega$ proves \eqref{eq:radius}.
\end{proof}

\subsection{Weighted geometric means}

\begin{lemma}\label{lem:mean}
Let $F,G>0$, let $a,b\geq0$ with $a+b=1$, and set $U=F^aG^b$. Then
\begin{equation}\label{eq:mean}
\frac{\Delta U}{U}
=
a\frac{\Delta F}{F}
+b\frac{\Delta G}{G}
-ab|\nabla\ln F-\nabla\ln G|^2.
\end{equation}
\end{lemma}

\begin{proof}
Use $\ln U=a\ln F+b\ln G$ and
$\Delta U/U=\Delta\ln U+|\nabla\ln U|^2$.
\end{proof}

\section{Proof of the Main Results}

Let $w>0$ be the first Dirichlet eigenfunction of $\Omega$:
\[
\Delta_gw=-\lambda_1(\Omega)w,
\qquad
w=0\quad\text{on }\partial\Omega.
\]
Fix $0<\alpha<\frac{4}{n-1}-\sigma$ and set
\[
U=f^{\frac{\sigma}{\sigma+\alpha}}
  w^{\frac{\alpha}{\sigma+\alpha}}.
\]
By Lemma \ref{lem:mean},
\[
(\sigma+\alpha)\frac{\Delta_gU}{U}
=
\sigma\frac{\Delta_gf}{f}
-\alpha\lambda_1(\Omega)
-\frac{\sigma\alpha}{\sigma+\alpha}
|\nabla\ln f-\nabla\ln w|^2.
\]
Hence
\begin{equation}\label{eq:combined-conf-ric}
\Ric_g^{U,\sigma+\alpha}
\geq
\Ric_g^{f,\sigma}+\alpha\lambda_1(\Omega)g
\geq
\bigl(\alpha\lambda_1(\Omega)-\kappa\bigr)g.
\end{equation}
If $\kappa\geq0$ and $\alpha\lambda_1(\Omega)\leq\kappa$, the following estimate is immediate. In every other case---in particular, automatically when $\kappa<0$---one has $\alpha\lambda_1(\Omega)-\kappa>0$, and Proposition \ref{prop:radius}, applied with $u=U$ and $\tau=\sigma+\alpha$, gives the same estimate after rearrangement:
\begin{equation}\label{eq:domain-gamma}
\lambda_1(\Omega)
\leq
\frac{\kappa}{\alpha}
+
\frac{\pi^2}{\alpha r_\Omega^2}
\left(
 n-1+\frac{(n-3)^2}{\frac{4}{\sigma+\alpha}-n+1}
\right).
\end{equation}

We now optimize \eqref{eq:domain-gamma}. Recall that
\[
D_\sigma=4-(n-1)\sigma.
\]
The elementary identity
\[
\frac{\sigma+\alpha}
{\alpha\left(\frac{4}{n-1}-\sigma-\alpha\right)}
=
\frac{(n-1)\sigma}{D_\sigma}\frac{1}{\alpha}
+
\frac{4}{D_\sigma}
\frac{1}{\frac{4}{n-1}-\sigma-\alpha}
\]
gives the decomposition
\[
\frac{\kappa}{\alpha}
+
\frac{\pi^2}{\alpha r_\Omega^2}
\left(
 n-1+\frac{(n-3)^2}{\frac{4}{\sigma+\alpha}-n+1}
\right)
=
\frac{A}{\alpha}
+\frac{B}{\frac{4}{n-1}-\sigma-\alpha},
\]
where
\[
A=
\kappa+
\frac{4\pi^2\bigl((n-1)-(n-2)\sigma\bigr)}{D_\sigma r_\Omega^2},
\qquad
B=
\frac{4\pi^2(n-3)^2}{(n-1)D_\sigma r_\Omega^2}.
\]
The estimate \eqref{eq:domain-gamma} for all $0<\alpha<\frac{4}{n-1}-\sigma$ forces $A\geq0$; otherwise its right-hand side would tend to $-\infty$ as $\alpha\downarrow0$, contradicting $\lambda_1(\Omega)>0$. Since $B\geq0$, the elementary identity
\[
\inf_{0<\alpha<\frac{4}{n-1}-\sigma}
\left(
\frac{A}{\alpha}
+\frac{B}{\frac{4}{n-1}-\sigma-\alpha}
\right)
=
\frac{n-1}{D_\sigma}(\sqrt A+\sqrt B)^2
\]
(with the corresponding endpoint limit in the degenerate cases $A=0$ or $B=0$) yields exactly \eqref{eq:domain-main}. This proves Theorem \ref{thm:domain}.

To prove the global estimate \eqref{eq:main-bottom}, choose a smooth exhaustion by connected relatively compact domains $\Omega_j$ such that, for a fixed $o\in M$,
\[
B_o(j)\Subset\Omega_j.
\]
Then $r_{\Omega_j}\geq j$. Fix $0<\alpha<\frac{4}{n-1}-\sigma$ and apply the unoptimized estimate \eqref{eq:domain-gamma} to $\Omega_j$. Letting $j\to\infty$ gives
\[
\lambda_1(-\Delta_g)\leq\frac{\kappa}{\alpha}.
\]
Finally, letting $\alpha\uparrow\frac{4}{n-1}-\sigma$ proves \eqref{eq:main-bottom}.

\section{The Four-Dimensional Q-Curvature Application}

The four-dimensional $Q$-curvature is
\[
Q_g=\frac{1}{6}\left(-\Delta_gR_g-3|\Ric_g|^2+R_g^2\right);
\]
see \cite{Branson1985}.  For a unit vector $X$, the Cauchy inequality gives
\[
|\Ric_g|^2
\geq
\Ric_g(X,X)^2+
\frac{(R_g-\Ric_g(X,X))^2}{3}
\geq
-\frac{2}{3}R_g\Ric_g(X,X)+\frac{1}{3}R_g^2.
\]
Equivalently, when $R_g>0$,
\[
\Ric_g(X,X)-\frac{\Delta_gR_g}{2R_g}
\geq3\frac{Q_g}{R_g}.
\]
This is the same pointwise estimate obtained from the Chang--Gursky--Yang inequality \cite[Lemma 1.2]{CGY}. The same conformal-Ricci inequality was recently used by Li \cite{Li2026} to establish Bonnet--Myers type results under lower bounds involving $Q$-curvature. Consequently,
\begin{equation}\label{eq:Q-confRic-sec4}
\Ric_g^{R,\frac{1}{2}}
\geq3\frac{Q_g}{R_g}g.
\end{equation}

If $\frac{Q_g}{R_g}\geq-\kappa$ on a bounded domain, then \eqref{eq:Q-confRic-sec4} gives
\[
\Ric_g^{R,\frac{1}{2}}\geq-3\kappa g.
\]
Applying the optimized estimate \eqref{eq:domain-main} with
\[
n=4,
\qquad
\sigma=\frac{1}{2},
\qquad
D_{1/2}=\frac{5}{2},
\]
and replacing $\kappa$ there by $3\kappa$, we obtain directly
\[
\lambda_1(\Omega)
\leq
\frac{1}{25}
\left[
\sqrt{90\kappa+\frac{96\pi^2}{r_\Omega^2}}
+\frac{4\pi}{r_\Omega}
\right]^2,
\]
which is \eqref{eq:Q-domain}. If the same assumptions hold globally on a complete noncompact four-manifold with $\kappa\geq0$, the global estimate \eqref{eq:main-bottom} gives \eqref{eq:Q-bottom}.

The quotient $Q/R$ has the same homogeneity as a Laplace eigenvalue under constant rescaling. The problem of finding a conformal metric with constant $Q/R$ was recently studied by Ge--Wang--Wei \cite{GeWangWei}. The inequality
\[
\inf_M\frac{Q_g}{R_g}
\leq-\frac{5}{18}\lambda_1(-\Delta_g)
\]
shows that positive bottom spectrum forces a quantitative negative part of $Q/R$ whenever $R_g>0$.

\section{A Spectral Ricci Extension of Cheng's Estimate}

\begin{proof}[Proof of Theorem \ref{thm:spectral-cheng}]
We first prove the bounded-domain estimate \eqref{eq:spectral-domain}. If $\alpha=0$, then
\[
\lambda_1\bigl(\Ric(x),\Omega\bigr)=\inf_{x\in\Omega}\Ric(x),
\]
and hence
\[
\Ric_g\geq\lambda_1\bigl(\Ric(x),\Omega\bigr)g
\quad\text{on }\Omega.
\]
Applying the unoptimized estimate \eqref{eq:domain-gamma} with a constant conformal factor, $\sigma=0$, and with the free parameter there replaced by $\beta$, gives \eqref{eq:spectral-domain}.

Now assume $0<\alpha<\frac{4}{n-1}$. Since $\Ric(x)$ is continuous, choose smooth functions $V_j$ on a neighborhood of $\overline\Omega$ such that
\[
\Ric(x)-\frac{1}{j}\leq V_j\leq\Ric(x).
\]
Let $u_j>0$ be the first Dirichlet eigenfunction of $-\alpha\Delta_g+V_j$ on $\Omega$:
\[
-\alpha\Delta_gu_j+V_j u_j
=
\lambda_1\bigl(-\alpha\Delta_g+V_j,\Omega\bigr)u_j.
\]
Since $V_j\leq\Ric(x)$, one has $\Ric_g\geq V_jg$, and therefore
\[
\begin{aligned}
\Ric_g^{u_j,\alpha}
&=\Ric_g-\alpha\frac{\Delta_gu_j}{u_j}g\\
&\geq
V_jg-
\left(V_j-\lambda_1\bigl(-\alpha\Delta_g+V_j,\Omega\bigr)\right)g\\
&=
\lambda_1\bigl(-\alpha\Delta_g+V_j,\Omega\bigr)g.
\end{aligned}
\]
Applying \eqref{eq:domain-gamma} with $f=u_j$, $\sigma=\alpha$, $\kappa=-\lambda_1(-\alpha\Delta_g+V_j,\Omega)$, and with the free parameter there replaced by $\beta$, we obtain
\[
\lambda_1\bigl(-\alpha\Delta_g+V_j,\Omega\bigr)
+\beta\lambda_1(\Omega)
\leq
\frac{\pi^2}{r_\Omega^2}
\left(
 n-1+
 \frac{(n-3)^2}{\frac{4}{\alpha+\beta}-n+1}
\right).
\]
Moreover,
\[
\lambda_1\bigl(-\alpha\Delta_g+\Ric(x),\Omega\bigr)-\frac{1}{j}
\leq
\lambda_1\bigl(-\alpha\Delta_g+V_j,\Omega\bigr)
\leq
\lambda_1\bigl(-\alpha\Delta_g+\Ric(x),\Omega\bigr).
\]
Letting $j\to\infty$ proves \eqref{eq:spectral-domain}.

Now suppose that $(M,g)$ is complete and noncompact. Choose a smooth connected exhaustion
\[
\Omega_1\Subset\Omega_2\Subset\cdots\Subset M,
\qquad
\bigcup_{j=1}^\infty\Omega_j=M,
\]
such that, for a fixed $o\in M$,
$
B_o(j)\Subset\Omega_j.
$
Then $r_{\Omega_j}\geq j$. By the variational definitions and domain monotonicity,
\[
\lambda_1(\Omega_j)\downarrow\lambda_1(-\Delta_g)
\]
and
\[
\lambda_1\bigl(-\alpha\Delta_g+\Ric(x),\Omega_j\bigr)
\downarrow
\lambda_1\bigl(-\alpha\Delta_g+\Ric(x)\bigr).
\]
Fix
\[
0<\beta<\frac{4}{n-1}-\alpha.
\]
Applying \eqref{eq:spectral-domain} to $\Omega_j$ and letting $j\to\infty$ gives
\[
\lambda_1\bigl(-\alpha\Delta_g+\Ric(x)\bigr)
+\beta\lambda_1(-\Delta_g)
\leq0.
\]
Finally, letting
\[
\beta\uparrow\frac{4}{n-1}-\alpha
\]
yields
\[
\lambda_1\bigl(-\alpha\Delta_g+\Ric(x)\bigr)
\leq
-\left(\frac{4}{n-1}-\alpha\right)\lambda_1(-\Delta_g)
\leq0,
\]
which is \eqref{eq:spectral-cheng}.
\end{proof}

\section*{Declaration on the use of AI-assisted tools}

The author used ChatGPT 5.6 Sol for assistance with routine calculations and minor editorial revisions. All mathematical arguments, computations, and statements have been independently reviewed and verified by the author, who takes full responsibility for the content of the manuscript.

\bigskip
\noindent Xiaoshang Jin\\
School of Mathematics and Statistics, Huazhong University of Science and Technology,\\
Wuhan, Hubei 430074, China.\\
Email: \text{jinxs@hust.edu.cn}\par

\end{document}